\documentclass[reqno]{amsart}

\usepackage[utf8]{inputenc}
\usepackage{amsmath}
\usepackage{amsthm}
\usepackage{amsfonts}
\usepackage{dsfont}
\usepackage{amssymb}
\usepackage{mathrsfs}
\usepackage[]{hyperref}
\usepackage{cite}
\usepackage{graphicx}
\usepackage{pgfplots}
\pgfplotsset{compat = newest}

\newtheorem{theorem}{Theorem}
\theoremstyle{plain}

\newtheorem{definition}{Definition}

\newtheorem{lemma}{Lemma}

\newtheorem{remark}{Remark}
\numberwithin{equation}{section}

\newcommand{\naturals}{\mathbb{N}}

\newcommand{\complexes}{\mathbb{C}}

\newcommand{\spec}[1]{\operatorname{\sigma}{\left(#1\right)}}

\newcommand{\tr}[1]{\operatorname{tr}{#1}}

\newcommand{\id}[1]{\mathrm{id}_{#1}}
\newcommand{\comm}[2]{\left[#1,#2\right]}
\newcommand{\acomm}[2]{\left\{#1,#2\right\}}

\newcommand{\hadj}{*}

\newcommand{\matr}[1]{\mathbb{M}_{#1}(\complexes)}
\newcommand{\matra}[2]{\mathbb{M}_{#1}(#2)}
\newcommand{\matrd}{\matr{d}}

\newcommand{\cpe}[1]{\operatorname{\mathsf{CP}}(#1)}
\newcommand{\dis}[1]{\operatorname{\mathrm{Dis}}(#1)}
\newcommand{\cdis}[1]{\operatorname{\mathrm{Dis}_\infty}(#1)}

\newcommand{\cocpe}[1]{\operatorname{\mathsf{coCP}}(#1)}

\newcommand{\dece}[1]{\operatorname{\mathsf{D}}(#1)}

\newcommand{\transpose}{\tau}
\newcommand{\unit}{\mathds{1}}

\begin{document}
\title[Decomposable semigroups on C*-algebras and D-divisible dynamical maps]{Decomposable semigroups on C*-algebras and D-divisible dynamical maps}
\author{Krzysztof Szczygielski}
\address[K. Szczygielski]
{Institute of Theoretical Physics and Astrophysics, Faculty of Mathematics, Physics and Informatics, University of Gda\'nsk, Wita Stwosza 57, 80-308 Gda\'nsk, Poland}%
\email[K. Szczygielski]{krzysztof.szczygielski@ug.edu.pl}

\begin{abstract}
We analyze semigroups of decomposable maps on C*-algebras in context of the algebraic structure of associated infinitesimal generators. Case of von Neumann algebras, including $B(\mathcal{H})$ for $\mathcal{H}$ a Hilbert space, is also addressed. We then elaborate on D-divisible (decomposably divisible) dynamical maps on the Banach space of trace class operators. Our analysis extends earlier results on decomposable dynamical maps on matrix algebras (J.~Phys.~A: Math.~Theor.~\textbf{56} 485202) and provides a partial generalization of the seminal work of Lindblad (Commun.~Math.~Phys.~\textbf{48} 119-130) on completely positive semigroups.
\end{abstract}

\maketitle

\section{Introduction}

It was shown in a seminal paper by G.~Lindblad \cite{Lindblad1976} that a family $(e^{tL})_{t\geqslant 0}$ of linear maps acting on a unital C*-algebra $\mathscr{A}$ is a uniformly continuous semigroup of completely positive (CP), unital maps if and only if its \emph{generator} $L$ is a \emph{completely dissipative} *-map on $\mathscr{A}$. Complete dissipativity means that $L(\unit)=0$ and, for every $n\in\naturals$, an extension $L_n = \id{n}\otimes L$ acting on a C*-algebra $\matra{n}{\mathscr{A}}$ is \emph{dissipative}, i.e.~satisfies inequality
\begin{equation}\label{eq:Cdiss}
    L_n (X^\hadj X) - L_n (X^\hadj)X - X^\hadj L_n (X) \geqslant 0
\end{equation}
for all $X\in\matra{n}{\mathscr{A}}$. Furthermore, $(e^{tL})_{t\geqslant 0}$ is a \emph{quantum dynamical semigroup} on a von Neumann algebra $\mathscr{A}$, i.e.~a uniformly continuous semigroup of weakly-* continuous unital CP maps on $\mathscr{A}$, if and only if $L$ is completely dissipative and weakly-* continuous. For a perhaps most robust case of $\mathscr{A}=B(\mathcal{H})$, a von Neumann algebra of all bounded linear operators on a Hilbert space $\mathcal{H}$, such $L$ admits a structure 
\begin{equation}\label{eq:StandardForm}
    L(A) = i\comm{H}{A} + \phi(A) - \frac{1}{2}\acomm{\phi(\unit)}{A}
\end{equation}
for all $A\in B(\mathcal{H})$, where $H \in B(\mathcal{H})$ is self-adjoint and $\phi : B(\mathcal{H})\to B(\mathcal{H})$ is CP and weakly-* continuous (here, $\comm{A}{B} = AB-BA$ and $\acomm{A}{B}=AB+BA$ respectively denote the standard commutator and anticommutator of operators; $\unit$ will always denote a unit in an algebra). Formula \eqref{eq:StandardForm} defines the most general form of a generator of uniformly continuous dynamical semigroup on $B(\mathcal{H})$ in dual (i.e.~Heisenberg) picture and is often referred to as the celebrated \emph{Lindblad \emph{(or \emph{standard})} form} in a literature, or interchangeably, as the \emph{GKSL \emph{(Gorini-Kossakowski-Sudarshan-Lindblad)} form} independently developed by Gorini \emph{et al.} \cite{Gorini1976} in matrix algebra case.

Recently, a generalization of the result of \cite{Gorini1976} was obtained in \cite{Szczygielski2023} in a case $\mathcal{H}=\complexes^d$, $B(\mathcal{H}) = \matrd$ under a less restrictive assumption that resulting semigroup is not specifically CP but rather \emph{decomposable}. In particular, by adopting proof techniques of \cite{Gorini1976} (see also \cite{Alicki2006,Breuer2002,Rivas2012}) it was shown that a semigroup $(e^{tL})_{t\geqslant 0}$ on $\matrd$ is unital and decomposable (actually trace preserving and decomposable) if and only if $L$ is of a form \eqref{eq:StandardForm} for Hermitian matrix $H$ and \emph{decomposable} map $\phi : \matrd\to\matrd$, i.e.~slightly more general than \eqref{eq:StandardForm}. The aim of this article is to extend results of \cite{Szczygielski2023} onto a general C*-algebraic case. In particular, we show correspondence between Lindblad's notion of dissipativity and decomposability of semigroups on C*-algebras and give a prescription for structure of the generator in a weakly-* continuous case of von Neumann algebras, similarly to \cite{Lindblad1976}.

The article is structured as follows. In the section \ref{sec:Background} we present a concise summary of necessary mathematical notions and previous results on the matter, including main results of the paper by Lindblad. Our results are divided into three main parts. Cases of both unital and nonunital semigroups on C*-algebras are treated respectively in sections \ref{sec:UnitalSemigroups} (Theorems \ref{thm:CstarDecomposition}--\ref{thm:WstarKrausForm}) and \ref{sec:NonunitalSemigroups} (Theorems \ref{thm:CstarDecompositionNonunital} and \ref{thm:NonunitalStandardForm}), while in the section \ref{sec:DdivQDM} we remark on the D-divisibility of quantum dynamical maps on $B_1 (\mathcal{H})$, the Banach space of trace class operators (Theorem \ref{thm:DynMapsStandardForm}).

\section{Background}\label{sec:Background}

\subsection{Positive maps}

As we will work with decomposable maps, we start with a recollection of notions regarding positive maps and their subclasses and, since our interest is focused on semigroups, we restrict to endomorphisms. A bounded, linear map $\phi : \mathscr{A}\to\mathscr{A}$ on a partially ordered C*-algebra $\mathscr{A}$ is called a \emph{positive map} if it maps a cone of positive elements into itself, $\phi(a^\hadj a) \geqslant 0$. If an extension $\phi_n = \id{n}\otimes\phi$ acting on a tensor C*-algebra $\matra{n}{\mathscr{A}}$ via $\phi_n([a_{ij}]) = [\phi(a_{ij})]$ is positive for some $n$, we say $\phi$ is \emph{$n$-positive}. If $\phi$ is $n$-positive for all $n\in\naturals$ we call such map \emph{completely positive} (CP); we will write $\phi\in\cpe{\mathscr{A}}$. Similarly, for $\transpose$ denoting the \emph{transposition} on $\matr{n}$, a map $\phi$ will be called \emph{$n$-copositive} if $\transpose\otimes\phi$ is positive for some $n$, and \emph{completely copositive} (coCP) if it is $n$-copositive for all $n\in\naturals$; we will denote this as $\phi\in\cocpe{\mathscr{A}}$. It is straightforward to notice that $\phi$ is $n$-copositive (resp.~coCP) if and only if $\transpose\circ\phi$ is $n$-positive (resp.~CP), where $\transpose$ in the latter means the \emph{transposition} on $\mathscr{A}$ represented as a C*-subalgebra of $B(\mathcal{H})$ for some Hilbert space $\mathcal{H}$ (we will use $\transpose$ symbol regardless of the underlying algebra to denote a transposition). We recall that a transposition $\transpose : B(\mathcal{H})\to B(\mathcal{H})$ may be defined as
\begin{equation}\label{eq:transpositionDef}
    \tau(A) = A^\tau = J A^\hadj J, \quad A\in B(\mathcal{H}),
\end{equation}
where $J : \mathcal{H}\to\mathcal{H}$ is an antilinear, idempotent isometry of complex conjugation of vectors in $\mathcal{H}$ with respect to a chosen orthonormal basis.
Finally, a bounded positive map $\varphi : \mathscr{A}\to\mathscr{A}$ is called \emph{decomposable}, or $\varphi\in\dece{\mathscr{A}}$, if it lays in the convex hull of $\cpe{\mathscr{A}}\cup\cocpe{\mathscr{A}}$, i.e.~may be expressed as
\begin{equation}\label{eq:DecMapDecomposition}
    \varphi = \phi_1 + \phi_2
\end{equation}
for some $\phi_1 \in \cpe{\mathscr{A}}$, $\phi_2\in\cocpe{\mathscr{A}}$. Woronowicz \cite{Woronowicz1976} showed all positive maps between matrix algebras $\matr{n}$ and $\matr{m}$ are decomposable when $mn\leqslant 6$, which fails to hold in generality.

Let $\mathscr{A}\subset B(\mathcal{H})$ for a Hilbert space $\mathcal{H}$. Dilation theorems of Stinespring \cite{Stinespring_1955} and Størmer \cite{Stormer_1982} imply that for any $\varphi\in \dece{\mathscr{A}}$ there exist Hilbert spaces $\mathcal{K}_1$, $\mathcal{K}_2$, *-homomorphism $\pi_1 : \mathscr{A}\to B(\mathcal{K}_1)$,  *-antihomomorphism $\pi_2 : \mathscr{A}\to B(\mathcal{K}_2)$ and linear operators $V_1 : \mathcal{H} \to \mathcal{K}_1$, $V_2 : \mathcal{H} \to \mathcal{K}_2$ such that
\begin{equation}\label{eq:DecDilation}
    \varphi (a) = V^{\hadj}_{1} \pi_1(a) V_1 + V^{\hadj}_{2} \pi_2(a) V_2, \quad a\in\mathscr{A},
\end{equation}
where we recovered both CP and coCP parts of $\varphi$, i.e.~$\phi_1 = V^{\hadj}_{1} \pi_1(\cdot) V_1$, $\phi_2 = V^{\hadj}_{2} \pi_2(\cdot) V_2$. Equivalently, $\varphi\in\dece{\mathscr{A}}$ if and only if $\varphi(a) = V^\hadj \pi(a) V$ for $V : \mathcal{H}\to \mathcal{K}_1 \oplus \mathcal{K}_2$, $V = V_1 + V_2$ and $\pi : \mathscr{A}\to B(\mathcal{K}_1 \oplus \mathcal{K}_2)$, $\pi = \pi_1 + \pi_2$ being a Jordan morphism.
\subsection{Completely positive semigroups}

Let now $(e^{tL})_{t\geqslant 0}$ be a one-parameter, uniformly continuous semigroup generated by a *-map $L\in B(\mathscr{A})$ on a unital C*-algebra $\mathscr{A}$. As we already mentioned in the Introduction, it was shown by Lindblad \cite[Theorem 1]{Lindblad1976} that all maps $e^{tL}$ are unital and CP if and only if $L$ is completely dissipative, i.e.~when $L(\unit) = 0$ and $L_n = \id{n}\otimes L$ satisfies dissipativity condition \eqref{eq:Cdiss} for all $X \in \matra{n}{\mathscr{A}}$ and all $n\in\naturals$. Further, $(e^{tL})_{t\geqslant 0}$ is a \emph{dynamical semigroup}, i.e.~a uniformly continuous semigroup of unital, CP and weakly-* continuous maps on a von Neumann algebra if and only if $L$ is completely dissipative and weakly-* continuous. Completely dissipative generators are further specified by two results. First \cite[Proposition 5]{Lindblad1976}, for any weakly-* continuous completely dissipative $L$ on a von Neumann algebra $\mathscr{A}$ there exists weakly-* continuous $\phi\in\cpe{\mathscr{A}}$ and a self-adjoint $H\in\mathscr{A}$ such that
\begin{equation}\label{eq:StandardForm2}
    L(a) = i\comm{H}{a} + \phi(a) - \frac{1}{2}\acomm{\phi(\unit)}{a}, \quad a\in\mathscr{A}.
\end{equation}
Second, the converse statement \cite[Proposition 6]{Lindblad1976} says that for any $\phi\in\cpe{\mathscr{A}}$ and $H\in\mathscr{A}$, $H^\hadj = H$, a *-map $L$ given by \eqref{eq:StandardForm2} is completely dissipative; here, it is only assumed that $\mathscr{A}$ is a C*-algebra, so this result has a broader scope. Finally, for $\mathscr{A}$ being a von Neumann algebra $B(\mathcal{H})$, application of all preceding results together with Kraus lemma yield \cite[Theorem 2]{Lindblad1976} that $(e^{tL})_{t\geqslant 0}$ is a uniformly continuous dynamical semigroup on $B(\mathcal{H})$ if and only if $L$ is completely dissipative and weakly-* continuous, which in turn happens if and only if
\begin{equation}
    L(A) = i\comm{H}{A} + \sum_n \left(V_{n}^{\hadj} A V_{n} - \frac{1}{2}\acomm{V_{n}^{\hadj} V_{n}}{A}\right)
\end{equation}
where $H=H^\hadj$ and $\sum_n V_{n}^{\hadj} V_{n}$ weakly-* converges in $B(\mathcal{H})$.

\section{Decomposable semigroups}\label{sec:DecSemigroups}
In this section we focus on semigroups of decomposable maps, with unital case treated in Section \ref{sec:UnitalSemigroups} and nonunital in Section \ref{sec:NonunitalSemigroups}. Some of our results will concern weakly-* continuous maps on von Neumann algebra, i.e.~generalization of quantum dynamical semigroups towards decomposable case -- naturally, we will call them \emph{decomposable dynamical semigroups} (see the definition below). We adopt a following notation from \cite{Lindblad1976}: for any set $\mathcal{X}$ of bounded linear maps on a von Neumann algebra $\mathscr{A}$ we denote with $\mathcal{X}_\sigma$ its subset consisting of maps continuous with respect to the weak-* topology on $\mathscr{A}$. It is known that $B(\mathscr{A})_\sigma$ is a uniformly closed subset of $B(\mathscr{A})$. From here onwards, all *-algebras are assumed unital.

\begin{definition}
    Let $\mathscr{A}$ be a von Neumann algebra. A uniformly continuous semigroup $(e^{tL})_{t\geqslant 0}$ of decomposable, unital weakly-* continuous maps on $\mathscr{A}$ will be called a decomposable dynamical semigroup.
\end{definition}

We stress here that the name \emph{decomposable dynamical semigroup} not necessarily implies such object can \emph{a priori} describe a realizable quantum evolution, since physicality of non-CP maps is disputable in general. We however adopt this name for convenience and rather focus on its mathematical properties as a generalization of a CP semigroup onto larger class of maps.

\subsection{General properties}

By Lumer-Phillips and Hille-Yosida theorems, the \emph{generator} $L\in B(\mathscr{A})$ of the semigroup is a *-map satisfying
\begin{equation}
    L = \lim_{t\to 0^+} \left\| L - t^{-1}(e^{tL} - \id{}) \right\|,
\end{equation}
i.e.~is a uniform derivative of $e^{tL}$ at $t=0$. In consequence, $L$ is necessarily weakly-* continuous.

The decomposition \eqref{eq:DecMapDecomposition} is highly nonunique in general. Unfortunately this fact poses a problem for our case: in particular, even though $e^{tL}$ is smooth in $t$ and decomposes into
\begin{equation}\label{eq:SemigroupDecomposition}
    e^{tL} = \phi_t + \transpose\circ\psi_t
\end{equation}
for some $\phi_t,\psi_t\in\cpe{\mathscr{A}}$ for all $t\geqslant 0$, nonuniqueness disallows us from claiming not only analytical properties, but even existence of well-defined functions $t\mapsto\phi_t$ and $t\mapsto\psi_t$. Therefore, some of our results will be explicitly formulated under the assumption that $\phi_t$ and $\psi_t$ exist as smooth functions; we will call $e^{tL}$ \emph{smoothly decomposable} in such case. We however suppose that it is possible, in principle, to show rigorously that such smooth decomposition can be shown to exist independently of the dimension, for any smooth curve in $B(\mathscr{A})$.

\begin{remark}\label{rem:TheRemark}
    The smooth decomposability requirement can be relaxed so that $\phi_t$ and $\psi_t$ appearing in decomposition \eqref{eq:SemigroupDecomposition} are only assumed to have right derivatives at $t=0$. Such alleviation does not influence our results, simultaneously allowing for much more freedom in selecting both the CP and coCP parts of the decomposition.
\end{remark}

For any *-algebra $\mathscr{X}$ and any $T \in B(\mathscr{X})$ let us define the \emph{dissipation function} $D_T : \mathscr{X}\to \mathscr{X}$ by
\begin{equation}\label{eq:DissFunction}
    D_T (x) = T(x^\hadj x) - T(x^\hadj)a - x^\hadj T(x), \quad x\in\mathscr{X},
\end{equation}
following Lindblad \cite{Lindblad1976}. Just like therein, a significant role will be played by maps enjoying $D_T (x) \geqslant 0$ for all $x\in\mathscr{X}$; such $T$ will be said to have a \emph{dissipation property}, or $T \in \dis{\mathscr{X}}$. Similarly, if $T_n = \id{n}\otimes T$ is in $\dis{\matra{n}{\mathscr{X}}}$ for all $n\in\naturals$, i.e.~every extension $T_n$ has a dissipation property, we will say $T$ has a \emph{complete dissipation property} and write $T \in \cdis{\mathscr{X}}$. Note that when $T$ has a dissipation property, then necessarily
\begin{equation}
    D_T(\unit) = -T(\unit) \geqslant 0, 
\end{equation}
so $T(\unit) \leqslant 0$. Evidently, in contrast to the approach exploited by Lindblad, we do not demand $\unit \in \ker{T}$ so classes $\dis{\mathscr{X}}$ and $\cdis{\mathscr{X}}$ are broader than classes of dissipative and completely dissipative maps, respectively, discussed extensively in \cite{Lindblad1976}.

If a positive map $\phi$ on a C*-algebra satisfies $\phi(\unit) \leqslant \unit$ we will call such map \emph{subunital}. Two following results, Lemmas \ref{prop:SubunitalCPsem} and \ref{lemma:ExpcoCP}, will be of importance:

\begin{lemma}\label{prop:SubunitalCPsem}
    Let $M$ be a *-map on a C*-algebra $\mathscr{A}$. If $M$ has a complete dissipation property and satisfies $M(\unit) \leqslant 0$ then it generates a semigroup $(e^{tM})_{t\geqslant 0}$ of CP, subunital maps on $\mathscr{A}$.
\end{lemma}

\begin{proof}
    Let us define a new map $G : \mathscr{A}\to\mathscr{A}$ by
    \begin{equation}
        M(a) = G(a) + \frac{1}{2}\acomm{M(\unit)}{a}, \quad a\in\mathscr{A}.
    \end{equation}
    Then, one checks with a bit of algebra that $G$ is a *-map, satisfies $G(\unit) = 0$ and is in $\cdis{\mathscr{A}}$, so it is completely dissipative in the sense of Lindblad. By \cite[Theorem 1]{Lindblad1976}, it generates a semigroup $(e^{tG})_{t\geqslant 0}$ of completely positive, unital maps on $\mathscr{A}$. On the other hand, a map $K = \frac{1}{2}\acomm{M(\unit)}{\cdot}$ can be easily checked to generate a semigroup
    \begin{equation}
        e^{tK}(a) = e^{\frac{1}{2}tM(\unit)} a e^{\frac{1}{2}tM(\unit)}, \quad a\in\mathscr{A},
    \end{equation}
    which in turn is completely positive by self-adjointness of $M(\unit)$. Let $M(\unit)$ be negative semidefinite, so that $\spec{M(\unit)}$ consists of nonpositive real numbers and let $E$ denote its spectral measure, so $M(\unit)= \int_{\spec{M(\unit)}}\lambda \, dE(\lambda)$ for $\lambda\leqslant 0$; then automatically $e^{tK}(\unit) = e^{tM(\unit)}$ satisfies
    \begin{equation}
        \unit - e^{tM(\unit)} = \int\limits_{\spec{M(\unit)}}(1-e^{t\lambda}) \, dE(\lambda) \geqslant 0
    \end{equation}
    i.e.~$e^{tK}$ is subunital. Using Lie-Trotter formula, we have therefore
    \begin{equation}
        e^{tM} = \lim_{n\to\infty} \left(e^{\frac{t}{n}G}e^{\frac{t}{n}K}\right)^{n}
    \end{equation}
    where each map $e^{\frac{t}{n}G}e^{\frac{t}{n}K}$ is CP and subunital; hence, the limit exists in norm as a CP subunital map by norm closedness of $\cpe{\mathscr{A}}$.
\end{proof}

\begin{lemma}\label{lemma:ExpcoCP}
    Let $\mathscr{A}$ be a C*-algebra and let $\psi\in\cpe{\mathscr{A}}$. Then $e^{t\,\transpose\circ\psi} \in \dece{\mathscr{A}}$.
\end{lemma}

\begin{proof}
There exists a CP map $\psi^\prime$ such that $\transpose\circ\psi = \psi^\prime \circ \transpose$. With easy algebra one confirms that
\begin{equation}
    (\transpose\circ\psi)^{2n} = (\psi^\prime\circ\psi)^n, \quad (\transpose\circ\psi)^{2n+1} = \transpose\circ\psi\circ (\psi^\prime\circ\psi)^n
\end{equation}
for all $n\in\naturals$, which leads to
\begin{equation}\label{eq:ExpcoCP}
    e^{t\,\transpose\circ\psi} = \sum_{n=0}^{\infty} \frac{t^{2n}}{(2n)!}(\psi^\prime\circ\psi)^n + \transpose\circ\psi\circ\sum_{n=0}^{\infty} \frac{t^{2n+1}}{(2n+1)!}(\psi^\prime\circ\psi)^{2n+1}
\end{equation}
after power series expanding and grouping terms for even and odd $n$. Both above series converge uniformly in $B(\mathscr{A})$ and, since $\psi,\psi^\prime\in\cpe{\mathscr{A}}$, represent completely positive maps on $\mathscr{A}$; consequently $e^{t\,\transpose\circ\psi} \in \dece{\mathscr{A}}$. 
\end{proof}

\subsection{Unital case}
\label{sec:UnitalSemigroups}

We first focus on the case of unital semigroups. Below we state the main result of this section, which specifies the link between decomposable semigroups and dissipation properties of their generators.

\begin{theorem}\label{thm:CstarDecomposition}
    Let $\mathscr{A}$ be a C*-algebra. The following statements hold:
    \begin{enumerate}
        \item \label{thm:CstarDecompositionItem1} Let $(e^{tL})_{t\geqslant 0}$ be a smoothly decomposable, uniformly continuous semigroup of unital maps on $\mathscr{A}$. Then $L = L_1 + L_2$, where $L(\unit) = 0$, $L_1 \in \cdis{\mathscr{A}}$ and $L_2 \in \cocpe{\mathscr{A}}$.
        \item \label{thm:CstarDecompositionItem2} Let $L_1 \in \cdis{\mathscr{A}}$, $L_2 \in \cocpe{\mathscr{A}}$. Then $(e^{tL})_{t\geqslant 0}$, $L = L_1 + L_2$, is a uniformly continuous semigroup of decomposable subunital maps on $\mathscr{A}$. If $L_1(\unit) + L_2 (\unit) = 0$ then $e^{tL}$ is unital.
    \end{enumerate}
\end{theorem}

\begin{proof}
Ad \ref{thm:CstarDecompositionItem1}. Let there exist families $(\phi_t)_{t\geqslant 0}\subset\cpe{\mathscr{A}}$, $(\psi_t)_{t\geqslant 0}\subset\cocpe{\mathscr{A}}$, both being smooth functions of $t$, such that
\begin{equation}
    e^{tL} = \phi_t + \transpose\circ\psi_t .
\end{equation}
Since $\phi_0 + \transpose\circ\psi_0 = \id{}$ is CP, we infer $\phi_0 = \id{}$, $\psi_0 = 0$. Both $\phi_t$, $\psi_t$ must admit uniform derivatives at $t=0$ (note Remark \ref{rem:TheRemark} about relaxing this condition), i.e.~there exist linear maps $R, S : \mathscr{A}\to\mathscr{A}$ such that
\begin{equation}\label{eq:RSlimits}
    \lim_{t\to 0^+} \left\| R - \frac{1}{t}(\phi_t - \id{}) \right\| = 0, \quad \lim_{t\to 0^+} \left\| S - \frac{1}{t}\psi_t \right\| = 0.
\end{equation}

Note that $\phi_t \in \cpe{\mathscr{A}}$ if and only if also $\id{n}\otimes\phi_t$ is completely positive on $\matra{n}{\mathscr{A}}$ for all $n\in\naturals$ and $t\geqslant 0$, which easily follows from Stinespring's dilation theorem \cite{Stinespring_1955}.
Maps $\phi_t$, $\psi_t$ are positive, satisfy $\phi_t(\unit) + \transpose\circ\psi_t (\unit) = \unit$, so
\begin{equation}
    \phi_t (\unit) \leqslant \unit, \quad \psi_t (\unit) \leqslant \unit
\end{equation}
for all $t\geqslant 0$, i.e.~are subunital. Then, $\phi_{n,t}=\id{n}\otimes\phi_t$, $\psi_{n,t}=\id{n}\otimes\psi_t$ are also subunital on $\matra{n}{\mathscr{A}}$. For a fixed $n$, denote
\begin{equation}
    f_X(t) = \phi_{n,t}(X^\hadj X) - \phi_{n,t}(X)^\hadj \phi_{n,t}(X).
\end{equation}
Any subunital CP map $\Phi$ on a C*-algebra satisfies $\|\Phi(\unit)\| \leqslant 1$ and $\Phi(\unit)^2 \leqslant \Phi(\unit)$ and so is subject to a version of Kadison-Schwarz inequality \cite{Paulsen2003}
\begin{equation}\label{eq:KSineq}
    \Phi(x^\hadj x) \geqslant \|\Phi(\unit)\| \Phi(x^\hadj x) \geqslant \Phi(x)^\hadj \Phi(x)
\end{equation}
for any $x$ in the algebra. Consequently, $f_X(t) \geqslant 0$ for all $X\in\matra{n}{\mathscr{A}}$, $n\in\naturals$ and $t\geqslant 0$. Smoothness of $\phi_t$ implies uniform differentiability of $f_X$ at $t=0$,
\begin{align}
    \frac{df_X}{dt}(0) &= \left.\frac{d\phi_{n,t}(X^\hadj X)}{dt}\right|_0 - \left.\frac{d\phi_{n,t}(X^\hadj)}{dt}\right|_0X - X^\hadj \left.\frac{d\phi_{n,t}(X)}{dt}\right|_0  \\
    &= R_{n}(X^\hadj X) - R_{n}(X^\hadj)X - X^\hadj R_{n}(X) \nonumber \\
    &= D_{R_{n}}(X) \nonumber
\end{align}
since $\phi_{n,0} = \id{}$, where we denoted $R_{n} = \id{n}\otimes R$. On the other hand $f_X(0) = 0$, so \eqref{eq:KSineq} yields $\frac{1}{t}\left(f_X(t)-f_X(0)\right) = \frac{1}{t}f_X(t)$ is positive for all $X\in\matra{n}{\mathscr{A}}$, $t\geqslant 0$ and remains positive in the limit $t\to 0^+$ since cone of positive elements is norm closed in $\matra{n}{\mathscr{A}}$, i.e.~$\frac{df_X}{dt}(0)= D_{R_{n}}(X) \geqslant 0$ for all $X$, $n$. This shows $R$ has a complete dissipation property.

As for $S$, note that $(\frac{1}{t}\psi_t)_{t\geqslant 0}$ is a net of CP maps on $\mathscr{A}$ converging to $S$ in norm; hence $S$ is also a CP map since $\cpe{\mathscr{A}}$ is closed in BW-topology and therefore norm closed in $B(\mathscr{A})$. Prescription \eqref{eq:transpositionDef} yields $\tau$ is a bounded, linear idempotent isometry on $\mathscr{A}$ and so $(\frac{1}{t}\psi_t)_{t\geqslant 0}$ converges to $S$ uniformly if and only if $(\transpose\circ \frac{1}{t}\psi_t)_{t\geqslant 0}$ converges to $\transpose\circ S$ uniformly. Therefore, using triangle inequality combined with \eqref{eq:RSlimits},
     \begin{equation}
         \lim_{t\to 0^+} \left\|R + \transpose\circ S - \frac{1}{t}(e^{tL}-\id{})\right\| = 0
     \end{equation}
and so $L = R + \transpose\circ S$. Now we identify $L_1 = R$, $L_2 = \transpose\circ S$ and note both maps preserve self-adjoint elements in $\mathscr{A}$, i.e.~are *-maps. $L(\unit)=0$ is obvious. This shows validity of the first statement of the Theorem.

Ad \ref{thm:CstarDecompositionItem2}. Take any $L_1\in \cdis{\mathscr{A}}$ and $L_2 \in \cocpe{\mathscr{A}}$. Then $e^{tL}$, for $L = L_1 + L_2$, may be expressed by Trotter product formula \cite{EngelNagel2000} as
\begin{equation}\label{eq:TrotterDec}
    e^{tL} = \lim_{n\to\infty} \left(e^{\frac{t}{n}L_1}e^{\frac{t}{n}L_2}\right)^{n}
\end{equation}
converging in norm for each $t\geqslant 0$. Since $L_1(\unit) \leqslant 0$, Lemma \ref{prop:SubunitalCPsem} yields $e^{\frac{t}{n}L_1}$ is CP and subunital, while $e^{\frac{t}{n}L_2}$ is decomposable by Lemma \ref{lemma:ExpcoCP}. This implies all maps appearing under the limit in \eqref{eq:TrotterDec} are decomposable, as is the limit $e^{tL}$ itself since $\dece{\mathscr{A}}$ is norm closed. Finally, condition $L_1(\unit) + L_2(\unit) = 0$ immediately implies unitality of the semigroup. This concludes the proof.
\end{proof}

\begin{theorem}\label{thm:WstarDecomposition}
    Let $\mathscr{A}$ be a von Neumann algebra. The following statements hold:
    \begin{enumerate}
        \item \label{thm:WstarDecompositionItem1} Let $(e^{tL})_{t\geqslant 0}$ be a smoothly decomposable dynamical semigroup on $\mathscr{A}$. Then $L = L_1 + L_2$, where $L(\unit) = 0$, $L_1 \in \cdis{\mathscr{A}}_\sigma$ and $L_2 \in \cocpe{\mathscr{A}}_\sigma$.
        \item \label{thm:WstarDecompositionItem2} Let $L_1 \in \cdis{\mathscr{A}}_\sigma$ and $L_2 \in \cocpe{\mathscr{A}}_\sigma$. Then $(e^{tL})_{t\geqslant 0}$, $L = L_1 + L_2$, is a uniformly continuous semigroup of decomposable, weakly-* continuous maps on $\mathscr{A}$. If $L_1(\unit) + L_2 (\unit) = 0$ then maps $e^{tL}$ are additionally unital, i.e.~$(e^{tL})_{t\geqslant 0}$ is a decomposable dynamical semigroup.
    \end{enumerate}
\end{theorem}

\begin{proof}
For part \ref{thm:WstarDecompositionItem1}, existence and properties of $L_1$ and $L_2$ follow from part \ref{thm:CstarDecompositionItem1} of Theorem \ref{thm:CstarDecomposition}. Subset $B(\mathscr{A})_\sigma$ of weakly-* continuous maps is norm closed in $B(\mathscr{A})$, so operators $R$ and $S$ appearing therein are weakly-* continuous as uniform derivatives, and so are $L_1$ and $L_2$. For statement \ref{thm:WstarDecompositionItem2}, application of part \ref{thm:CstarDecompositionItem2} of Theorem \ref{thm:CstarDecomposition} yields $e^{tL}$ is decomposable. When $L_1$, $L_2$ are assumed weakly-* continuous, so are the maps appearing under the limit in \eqref{eq:TrotterDec}. Then, closedness of $B(\mathscr{A})_\sigma$ yields weak-* continuity of $e^{tL}$ as a uniform limit in the Trotter formula.
\end{proof}

\begin{theorem}\label{thm:WstarStandardForm}
    Let $\mathscr{A}$ be a von Neumann algebra. Set $L = L_1 + L_2$, where
    \begin{enumerate}
        \item $L_1\in \cdis{\mathscr{A}}_\sigma$,
        \item $L_2\in\cocpe{\mathscr{A}}_\sigma$,
        \item $L_1(\unit) + L_2(\unit) = 0$.
    \end{enumerate}
    Then there exists $\varphi\in\dece{\mathscr{A}}_\sigma$ and $H\in\mathscr{A}$, $H^\hadj = H$, such that
    \begin{equation}\label{eq:StandardForm3}
        L(a) = i\comm{H}{a} + \varphi(a) - \frac{1}{2}\acomm{\varphi (\unit)}{a}, \quad a\in\mathscr{A}.
    \end{equation}
\end{theorem}

\begin{proof}
Following the same idea as in the proof of Lemma \ref{prop:SubunitalCPsem}, let us define a new map $G\in B(\mathscr{A})_\sigma$ by
\begin{equation}
    L_1(a) = G (a) + \frac{1}{2}\acomm{L_1 (\unit)}{a}
\end{equation}
which once again is checked to be completely dissipative in the sense of Lindblad. Application of \cite[Proposition 5]{Lindblad1976} yields existence of $\phi \in \cpe{\mathscr{A}}_\sigma$ and a self-adjoint $H\in\mathscr{A}$ such that
\begin{equation}
    L_1 (a) = i \comm{H}{a} + \phi(a) - \frac{1}{2}\acomm{\phi(\unit)}{a} + \frac{1}{2}\acomm{L_1 (\unit)}{a}.
\end{equation}
$L_2$ may be then expressed as $L_2 = \transpose\circ\psi$ for some $\psi\in\cpe{\mathscr{A}}_\sigma$. Condition $L_1(\unit) + L_2(\unit) = 0$ then yields $L_1(\unit) = -\transpose \circ \psi(\unit)$ which puts $L$ into the claimed form \eqref{eq:StandardForm3} for $\varphi = \phi + \transpose\circ\psi$ weakly-* continuous.
\end{proof}

\begin{theorem}\label{thm:CstarStandardForm}
    Let $\mathscr{A}$ be a C*-algebra. If $\varphi\in\dece{\mathscr{A}}$ and $H\in\mathscr{A}$ is self-adjoint then there exist maps $L_1, L_2 \in B(\mathscr{A})$ satisfying conditions
        \begin{enumerate}
        \item $L_1\in \cdis{\mathscr{A}}$,
        \item $L_2\in\cocpe{\mathscr{A}}$,
        \item $L_1(\unit) + L_2(\unit) = 0$,
    \end{enumerate}
    such that $L = L_1 + L_2$ is of a form \eqref{eq:StandardForm3}.
\end{theorem}

\begin{proof}
Let $\varphi = \phi + \transpose\circ\psi$ for some $\phi,\psi\in\cpe{\mathscr{A}}$ and take $H = H^\hadj\in\mathscr{A}$. Define $L_1, L_2 \in B(\mathscr{A})$ by
    \begin{equation}
        L_1 (a) = i \comm{H}{a} + \phi(a) - \frac{1}{2}\acomm{\varphi(\unit)}{a},
    \end{equation}
    \begin{equation}
        L_2 (a) = \transpose\circ\psi .
    \end{equation}
Then automatically $L_2$ is coCP and $L_1(\unit)+L_2(\unit)=0$. It remains to show $L_1$ has a complete dissipation property. Let $L_{1,n} = \id{n}\otimes L_1$. With easy algebra one finds, for any $X\in\matra{n}{\mathscr{A}}$, $n\in\naturals$,
\begin{equation}
    D_{L_{1,n}}(X) = \phi_n(X^\hadj X) - \phi_n(X^\hadj)X - X^\hadj \phi_n (X) + X^\hadj \psi_n(\unit) X
\end{equation}
where $\phi_n = \id{n}\otimes\phi$. Employing decomposition \eqref{eq:DecDilation} for $\varphi_n$, we obtain, with $V_{i,n} = \unit\otimes V_{i}$ and $\pi_{i,n} = \id{n}\otimes\pi_i$, $i=1,2$,
\begin{align}
    D_{L_{1,n}}(X) &= V_{1,n}^{\hadj} \pi_{1,n} (X^\hadj X) V_{1,n} - V_{1,n}^{\hadj} \pi_{1,n} (X^\hadj) V_{1,n} X \\
    &- X^\hadj V_{1,n}^{\hadj} \pi_{1,n} (X) V_{1,n} + X^\hadj V_{1,n}^{\hadj} V_{1,n} X + X^\hadj V_{2,n}^{\hadj} V_{2,n} X \nonumber \\
    &= \left(\pi_{1,n} (X)V_{1,n} - V_{1,n} X\right)^\hadj \left(\pi_{1,n} (X)V_{1,n} - V_{1,n} X\right) \nonumber \\
    &+ \left(V_{2,n} X\right)^\hadj \left(V_{2,n} X\right), \nonumber
\end{align}
so $D_{L_{1,n}}(X) \geqslant 0$ and $L_1 \in \cdis{\mathscr{A}}$.
\end{proof}

Theorems \ref{thm:WstarDecomposition}, \ref{thm:WstarStandardForm} and \ref{thm:CstarStandardForm} combined with the Kraus lemma yield the following result, which determines the most general form of a generator of a decomposable dynamical semigroup on $B(\mathcal{H})$. This is a direct generalization of \cite[Theorem 2]{Lindblad1976} onto the decomposable case, with the restriction of smooth decomposability condition.

\begin{theorem}\label{thm:WstarKrausForm}
    The following statements hold:
    \begin{enumerate}
        \item If $(e^{tL})_{t\geqslant 0}$ is a smoothly decomposable dynamical semigroup on $B(\mathcal{H})$ then $L$ admits a structure
        \begin{equation}\label{eq:DecL}
            L(A) = i\comm{H}{A} + \sum_{n=1}^{\infty} \left( V_{n}^{\hadj} A V_n + W_{n}^{\hadj} A^\transpose W_n - \frac{1}{2}\acomm{V_{n}^{\hadj} V_n + W_{n}^{\hadj} W_n}{A} \right),
        \end{equation}
        for $H=H^\hadj \in B(\mathcal{H})$ and two families $(V_n)_{n\in\naturals}, (W_n)_{n\in\naturals} \subset B(\mathcal{H})$ such that $\sum_{n=1}^{\infty} V_{n}^{\hadj} V_n$ and $\sum_{n=1}^{\infty} W_{n}^{\hadj} W_n$ weakly-* converge in $B(\mathcal{H})$.
        \item Conversely, if $L$ is of a form \eqref{eq:DecL} then $(e^{tL})_{t\geqslant 0}$ is a decomposable dynamical semigroup on $\mathscr{A}$.
    \end{enumerate}
\end{theorem}

\subsection{Nonunital case}
\label{sec:NonunitalSemigroups}

Here we note on a slight generalization of previous results by not restricting ourselves to unital maps, i.e.~we allow $e^{tL}(\unit)\neq \unit$. Similarly to \cite{Lindblad1976} this is achieved by perturbing a generator of unital semigroup and provides an extension of \cite[Corollary 2]{Lindblad1976}.

\begin{theorem}\label{thm:CstarDecompositionNonunital}
    Let $\mathscr{A}$ be a C*-algebra. The following statements hold:
    \begin{enumerate}
        \item\label{thm:CstarDecompositionNonunitalItem1} If $(e^{tL})_{t\geqslant 0}$ is a smoothly decomposable semigroup on $\mathscr{A}$ then $L = L_1 + L_2$ such that $L_2\in\cocpe{\mathscr{A}}$ and $L_{1,n} = \id{n}\otimes L_1$ is a *-map which satisfies inequality
        \begin{equation}\label{eq:NonUnitalLineq}
            L_{1,n}(X^\hadj X) - L_{1,n}(X^\hadj)X - X^\hadj L_{1,n}(X) + X^\hadj L_{1,n}(\unit) X \geqslant 0
        \end{equation}
        for every $X\in\matra{n}{\mathscr{A}}$ and every $n\in\naturals$.
        \item\label{thm:CstarDecompositionNonunitalItem2} If a *-map $L_1\in B(\mathscr{A})$ satisfies \eqref{eq:NonUnitalLineq} for all $X\in\matra{n}{\mathscr{A}}$, $n\in\naturals$ and $L_2\in\cocpe{\mathscr{A}}$ then $e^{tL}\in\dece{\mathscr{A}}$ for $L = L_1+L_2$.
    \end{enumerate}
\end{theorem}

\begin{proof}
Ad \ref{thm:CstarDecompositionNonunitalItem1}. Let again $e^{tL} = \phi_t + \transpose\circ\psi_t$ for $\phi_t\in\cpe{\mathscr{A}}$, $\psi_t\in\cocpe{\mathscr{A}}$ smooth in $t$, $\phi_0 = \id{}$, $\psi_0 = 0$ and denote again $L_1 = \left.\frac{d\phi_t}{dt}\right|_0$, $L_2 = \left.\frac{d\psi_t}{dt}\right|_0$, both *-maps on $\mathscr{A}$. Then $L_2$ is coCP by closedness argument like in a proof of Theorem \ref{thm:CstarDecomposition}. For any $n\in\naturals$, take $X\in\matra{n}{\mathscr{A}}$ and define
\begin{equation}\label{eq:fXnonunitalDef}
    f_X(t) = \phi_{n,t}(X^\hadj X) - \phi_{n,t}(X)^\hadj \phi_{n,t}(\unit)^{-1}\phi_{n,t}(X),
\end{equation}
where $\phi_{n,t} = \id{n}\otimes\phi_t$. Map $\phi_t$, being nonunital in principle, is subject to a generalized Kadison-Schwarz inequality
\begin{equation}\label{eq:KSineqNonunital}
    \phi_{t}(a^\hadj a) \geqslant \phi_{t}(a)^\hadj \phi_{t}(\unit)^{-1}\phi_{t}(X),
\end{equation}
which then extend to $\phi_{n,t}$ for all $n$, hence $f_X(t) \geqslant 0$. Differentiating $f_X$ like in Theorem \ref{thm:CstarDecomposition} we obtain $\frac{df_X}{dt}(0)$ is given by the left hand side of \eqref{eq:NonUnitalLineq}. Again we have $f_X(0) = 0$ so $\frac{1}{t}\left( f_X(t) - f_X (0)\right)$ is positive for all $t\geqslant 0$ by \eqref{eq:KSineqNonunital} and remains positive in the limit $t\to 0^+$. In consequence $\frac{df_X}{dt}(0) \geqslant 0$ for all $X$, $n$ which is the inequality \eqref{eq:NonUnitalLineq}.

Ad \ref{thm:CstarDecompositionNonunitalItem2}. Let $L_1$, $L_2$ be as assumed. Set $G : \mathscr{A}\to \mathscr{A}$ by
\begin{equation}
    G(a) = L_1(a) - \frac{1}{2}\acomm{L_1(\unit)}{a}
\end{equation}
and extend it to $G_n = \id{n}\otimes G$, $n\in\naturals$. Then, one can check with direct algebra that $G(\unit) = 0$ and $D_{G_n}(X)$ is again the left hand side of \eqref{eq:NonUnitalLineq}; hence $D_{G_n}(X)\geqslant 0$. In result, $G$ is completely dissipative in the sense of \cite{Lindblad1976} and generates a semigroup of CP unital maps on $\mathscr{A}$. Map $\frac{1}{2}\acomm{L_{1} (\unit)}{\cdot}$ is then easily shown to generate a semigroup $a \mapsto e^{\frac{1}{2}tL_1(\unit)} a e^{\frac{1}{2}tL_1(\unit)}$ which is CP since $L_1(\unit)$ is self-adjoint. Using Lie-Trotter formula,
\begin{equation}
    e^{tL_1} = \lim_{n\to\infty} \left(e^{\frac{t}{n}G}e^{\frac{t}{2n}\acomm{L_1 (\unit)}{\cdot}}\right)^n
\end{equation}
is CP as a uniform limit of CP maps in the closed cone $\cpe{\mathscr{A}}$. Employing the Lie-Trotter formula again, this time to $e^{t(L_1 + L_2)}$, yields it to be decomposable since $e^{tL_2}\in\dece{\mathscr{A}}$ by Lemma \ref{lemma:ExpcoCP}. The proof is concluded.
\end{proof}

\begin{theorem}\label{thm:NonunitalStandardForm}
    The following statements hold:
    \begin{enumerate}
        \item\label{thm:NonunitalStandardFormItem1} Let $\mathscr{A}$ be a C*-algebra. If $\varphi\in\dece{\mathscr{A}}$, $K\in\mathscr{A}$ and $L\in B(\mathscr{A})$ is a *-map defined via
        \begin{equation}\label{eq:LnonunitalStandardForm}
            L(a) = \varphi(a) + Ka + aK^\hadj
        \end{equation}
        then $e^{tL}\in\dece{\mathscr{A}}$.
        \item\label{thm:NonunitalStandardFormItem2} Let $\mathscr{A}$ be a von Neumann algebra. If $(e^{tL})_{t\geqslant 0}$ is smoothly decomposable semigroup of weak-* continuous maps on $\mathscr{A}$ then there exists $\varphi\in\dece{\mathscr{A}}_\sigma$ such that $L$ is of a form \eqref{eq:LnonunitalStandardForm}.
    \end{enumerate}
\end{theorem}

\begin{proof}
Ad \ref{thm:NonunitalStandardFormItem1}. Take $K\in\mathscr{A}$ and $\varphi = \phi + \transpose\circ\psi$ for some $\phi,\psi \in \cocpe{\mathscr{A}}$. Denote
\begin{equation}
    \Psi(a) = \phi(a) + Ka + aK^\hadj.
\end{equation}
Then $e^{t\Psi}$ is CP by \cite[Theorem 3]{Lindblad1976}. With Lie-Trotter formula,
\begin{equation}
    e^{tL} = \lim_{n\to\infty}\left( e^{\frac{t}{n}\Psi} e^{\frac{t}{n} \transpose\circ\psi} \right)^{n}
\end{equation}
is then decomposable since $e^{t\,\transpose\circ\psi}\in\dece{\mathscr{A}}$ via Lemma \ref{lemma:ExpcoCP}.

Ad \ref{thm:NonunitalStandardFormItem2}. Part \ref{thm:CstarDecompositionNonunitalItem1} of Theorem \ref{thm:CstarDecompositionNonunital} yields $L = L_1 + L_2$ for $L_2\in\cocpe{\mathscr{A}}_\sigma$ and $L_1 \in B(\mathscr{A})_\sigma$ satisfies inequality \eqref{eq:NonUnitalLineq} for all $X\in\matra{n}{\mathscr{A}}$, $n\in\naturals$. Define
\begin{equation}
    M(a) = L_1 (a) - \frac{1}{2}\acomm{L_1 (\unit)}{a}, \quad a\in\mathscr{A}
\end{equation}
and $M_n = \id{n}\otimes M$. Then $M(\unit) = 0$ and $D_{M_n}(X) = D_{L_{1,n}}(X) + X^\hadj L_1(\unit)X \geqslant 0$, i.e.~$M$ is completely dissipative in the sense of \cite{Lindblad1976} and it generates a semigroup $(e^{tM}_{t\geqslant 0})$ of CP unital maps. By \cite[Theorem 3]{Lindblad1976} it then admits a form
\begin{equation}
    M(a) = \phi(a) + ka + ak^\hadj
\end{equation}
for some $k\in\mathscr{A}$ and $\phi\in\cpe{\mathscr{A}}_\sigma$. As $L_2 = \transpose\circ\psi$ for $\psi\in\cpe{\mathscr{A}}_\sigma$, we have
\begin{align}
    L(a) &= M(a) + L_2 (a) + \frac{1}{2}\acomm{L_1(\unit)}{a}\\
    &= \phi(a) + \psi(a)^\transpose + \left(k + \frac{1}{2}L_1(\unit)\right)a + a \left(k + \frac{1}{2}L_1(\unit)\right)^\hadj\nonumber
\end{align}
is of a form \eqref{eq:LnonunitalStandardForm} for $\varphi = \phi + \transpose\circ\psi$ weakly-* continuous and $K = k + \frac{1}{2}L_1(\unit)$.
\end{proof}

\section{D-divisible quantum dynamics}
\label{sec:DdivQDM}

In this section we assume $\mathscr{A} = B(\mathcal{H})$ (with $\unit$ being the identity operator). Its predual is then $B_1 (\mathcal{H})$, the Banach space of all trace class operators on $\mathcal{H}$ complete with the trace norm $\|\cdot\|_1$.

A family $(\Lambda_t)_{t\geqslant 0}$ of bounded maps on $B_1 (\mathcal{H})$ will be called a \emph{(quantum) dynamical map} if $\Lambda_t$ is positive and trace preserving for all $t\geqslant 0$ and $\Lambda_t (\rho) \to \rho$ in trace norm for every $\rho\in B_1(H)$ as $t\to 0^+$. These conditions are equivalent to the dual $\Lambda_{t}^{\prime}$ being a positive, unital, weakly-* continuous map on $B(\mathcal{H})$, satisfying $\Lambda_{t}^{\prime}(A)\to A$ in weak-* operator topology for all $A\in B(\mathcal{H})$ as $t\to 0^+$. Traditionally, when we set $\rho_t = \Lambda_t (\rho)$ for some $\rho\geqslant 0$, $\tr{\rho}=1$, we interpret $\rho_t$ as a time-dependent \emph{density operator} of a given physical system, and so $\Lambda_t$ encodes a physical evolution of a mixed state. We restrict attention exclusively to dynamical maps which are subject to a Master Equation
\begin{equation}\label{eq:ME}
    \frac{d\rho_t}{dt} = \mathcal{L}_t (\rho_t) 
\end{equation}
or, equivalently $\frac{d\Lambda_t}{dt} = \mathcal{L}_t \circ\Lambda_t$, for a time-dependent *-map $\mathcal{L}_t \in B(B_1(\mathcal{H}))$ satisfying $\tr{\mathcal{L}_t(\rho)}=0$ for all $\rho\in B_1(\mathcal{H})$. Its dual $\mathcal{L}_t \in B(B(\mathcal{H}))$ is then a weakly-* continuous map satisfying $L_t (\unit)=0$, i.e.~is to be understood as in the previous sections. Clearly, both $L_t$ and $\mathcal{L}_t$ are referred to as generator, in the Heisenberg and Schr\"{o}dinger picture, respectively. A dynamical map subject to \eqref{eq:ME} is \emph{divisible}, i.e.~there exists a two-parameter family $(V_{t,s})_{t\geqslant s}$ of \emph{propagators}, also trace preserving, such that $\Lambda_t = V_{t,s}\circ\Lambda_s$ for all $s \in [0,t]$. When $V_{t,s}$ is always an $n$-positive map we say $(\Lambda_t)_{t\geqslant 0}$ is \emph{$n$-divisible}. In the most prominent case when $V_{t,s} \in \cpe{B_1(\mathcal{H})}$ we call the dynamical map \emph{CP-divisible} or \emph{Markovian}.

The structure of $\mathcal{L}_t$ for CP-divisible invertible dynamical maps has been long known to be dual to the celebrated GKSL form \eqref{eq:StandardForm}. When the constraint of CP-divisibility is alleviated by requiring $V_{t,s}$ to be decomposable instead, we call the dynamical map \emph{D-divisible}. In case of finite-dimensional Hilbert space, i.e.~when $B_1(\mathcal{H}) \simeq \matrd$, it was shown that $\mathcal{L}_t$ retains its general Lindblad-like structure, however with $\phi$ appearing in \eqref{eq:StandardForm} being decomposable \cite{Szczygielski2023}. The following theorem extends results of \cite{Szczygielski2023} onto the case of a general Hilbert space.

\begin{theorem}\label{thm:DynMapsStandardForm}
Let $(\Lambda_t)_{t\geqslant 0}$ be a dynamical map on $B_1(\mathcal{H})$ satisfying time-local Master Equation \eqref{eq:ME} and $\Lambda_0 = \id{}$. Then:
\begin{enumerate}
    \item\label{thm:DynMapsStandardFormItem1} If $V_{t,s}$ is smoothly decomposable for all $t\geqslant s$ then there exists a self-adjoint $H_t\in B(\mathcal{H})$ and $\Phi_t \in \dece{B_1(\mathcal{H})}$ such that
    \begin{equation}\label{eq:LtDdiv}
        \mathcal{L}_t (\rho) = -i\comm{H_t}{\rho} + \Phi_t (\rho) - \frac{1}{2}\acomm{\Phi_{t}^{\prime}(\unit)}{\rho}
    \end{equation}
    for all $\rho\in B_1(\mathcal{H})$.
    \item\label{thm:DynMapsStandardFormItem2} Conversely, when $\mathcal{L}_t$ admits a form \ref{eq:LtDdiv} then $V_{t,s}\in\dece{B_1(\mathcal{H})}$ for all $t\geqslant s$.
\end{enumerate}
\end{theorem}

\begin{proof}
For part \ref{thm:DynMapsStandardFormItem1}, let $V_{t,s}$ be decomposable, $V_{t,s} = Z_{t,s} + \transpose\circ Y_{t,s}$, for $Z_{t,s},Y_{t,s}\in\cpe{B_1(\mathcal{H})}$ smooth in $t\geqslant s\geqslant 0$. Again, we have $V_{t,t} = \id{}$ so $Z_{t,t} = \id{}$ and $Y_{t,t}=0$. Then, for every $t$ there exist maps $\mathcal{R}_t, \mathcal{S}_t \in B(B_1 (\mathcal{H}))$ satisfying
\begin{equation}
    \lim_{h\to 0^+}\left\| \mathcal{R}_t - \frac{1}{h}\left(Z_{t+h,t}-\id{}\right) \right\| = 0, \quad \lim_{h\to 0^+}\left\| \mathcal{S}_t - \frac{1}{h}Y_{t+h,t} \right\| = 0,
\end{equation}
i.e.~uniform derivatives of $Z_{t,s}$ and $Y_{t,s}$ at $t=s$. This however shows
\begin{equation}
    \lim_{h\to 0^+}\left\| \mathcal{R}_t + \transpose \circ \mathcal{S}_t - \frac{1}{h}\left(V_{t+h,t}-\id{}\right) \right\| = 0
\end{equation}
by triangle inequality; since $V_{t,s}$ enjoys differential equation
\begin{equation}
    \frac{d\Lambda_t}{dt} = \lim_{h\to 0^+}\frac{V_{t+h,t}-\id{}}{h}\circ\Lambda_t = \mathcal{L}_t\circ\Lambda_t,
\end{equation}
we infer $\mathcal{L}_t = \mathcal{R}_t + \transpose \circ \mathcal{S}_t$, or $L_t = R_{t} + \transpose\circ S_{t}$ for maps $R_{t}$, $S_{t} : B(\mathcal{H})\to B(\mathcal{H})$ dual to $\mathcal{R}_{t}$, $\mathcal{S}_t$. Since $\left(\frac{1}{h}Y_{t+h,t}\right)_{h\geqslant 0}$ is a convergent net of CP maps, $\mathcal{S}_t$ is also CP since $\cpe{B_1 (\mathcal{H})}$ is closed (in supremum norm induced by trace norm); hence $S_t \in \cpe{B(\mathcal{H})}$. Now, let $\phi_{t,s}$, $\psi_{t,s}$ be maps on $B(\mathcal{H})$ dual to $Z_{t,s}$, $Y_{t,s}$ and define
\begin{equation}
    f_X (t,s) = \phi_{n;t,s}(X^\hadj X) - \phi_{n;t,s}(X^\hadj)\phi_{n;t,s}(X)
\end{equation}
for any $X \in \matra{n}{B(\mathcal{H})}$, $n\in\naturals$ where we denoted $\phi_{n;t,s} = \id{n}\otimes\phi_{t,s}$ and $\psi_{n;t,s} = \id{n}\otimes\psi_{t,s}$. With direct calculation,
\begin{align}\label{eq:DmapfXder}
    \frac{\partial f_X}{\partial t}(s,s) &= \left.\frac{\partial \phi_{n;t,s}(X^\hadj X)}{\partial t}\right|_s - \left.\frac{\partial \phi_{n;t,s}(X^\hadj)}{\partial t}\right|_s X - X^\hadj \left.\frac{\partial \phi_{n;t,s}(X)}{\partial t}\right|_s \\
    &= R_{n,s}(X^\hadj X) - R_{n,s}(X^\hadj )X - X^\hadj R_{n,s}(X). \nonumber
\end{align}
On the other hand, since $f_X(t,t) = 0$ we have $\frac{1}{h}\left(f_X(t+h,t)-f_X(t,t)\right) = \frac{1}{h}f_X(t+h,t) \geqslant 0$ and so \eqref{eq:DmapfXder} is nonnegative for all $X$, $n$ and $s\geqslant 0$, i.e.~$R_t$ has a complete dissipation property for all $t$. Trace preservation of $\Lambda_t$ yields $L_t(\unit) = R_t (\unit) + S_t (\unit)^\transpose = 0$, so maps $R_t$, $S_t$ satisfy all the assumptions of Theorem \ref{thm:WstarStandardForm} for all $t\geqslant 0$ after identification $L_1 = R_t$, $L_2 = \transpose\circ S_t$, and hence there exists a self-adjoint $H_t \in B(\mathcal{H})$ and a decomposable map $\varphi_t \in B(B(\mathcal{H}))$ such that $L_t$ is of a form \eqref{eq:StandardForm3}. It is then a simple exercise to show its predual $\mathcal{L}_t$ must be of a form \eqref{eq:LtDdiv} with $\Phi_t$ predual to $\varphi_t$.

For part \ref{thm:DynMapsStandardFormItem2}, note that when $\mathcal{L}_t$ is as in \eqref{eq:LtDdiv} then $L_t$ is of a form \eqref{eq:StandardForm3}, with a self-adjoint $H_t$ and decomposable $\varphi_t$ replacing $H$ and $\varphi$, respectively. Theorems \ref{thm:WstarDecomposition} and \ref{thm:CstarStandardForm} then yield $L_t$ generates a norm continuous semigroup $(e^{sL_t})_{s\geqslant 0}$ of decomposable unital maps on $B(\mathcal{H})$, for every $t$. Consequently, $(e^{s\mathcal{L}_t})_{s\geqslant 0}$ is a uniformly continuous semigroup of weakly-* continuous, decomposable trace preserving maps on $B_1(\mathcal{H})$. Then, the \emph{time-splitting formula} \cite{Rivas2012} yields
\begin{equation}
    V_{t,s} = \lim_{\max{|t_{j+1}-t_j|}\to 0}\prod_{j=n-1}^{0} e^{(t_{j+1}-t_j)\mathcal{L}_{t_j}}, \quad t = t_n \geqslant ... \geqslant t_0 = s,
\end{equation}
is a limit of a sequence of decomposable and trace preserving maps. This concludes the proof.
\end{proof}

\section{Summary}

We investigated semigroups of decomposable positive maps on C*-algebras with special attention granted to algebraic structure of their infinitesimal generators, with additional assumption of smooth decomposability. Our results provide an extension of the work of Lindblad \cite{Lindblad1976} onto broader class of decomposable maps. It was shown that properties of both the semigroup and its generator resemble completely positive case explored extensively in \cite{Lindblad1976} and may be treated in quite a similar manner. A notion of D-divisible quantum dynamical maps on $B_1 (\mathcal{H})$ was formulated for a general Hilbert space $\mathcal{H}$, effectively extending previous results from matrix algebras \cite{Szczygielski2023}. It remains unclear if such a strong condition of \emph{smooth} decomposability is really necessary for key theorems to apply; we suspect that it might not be the case or that this requirement can be at least weakened. A natural open question arises from our analysis which concerns determining internal structure of generators of \emph{indecomposable} semigroups, perhaps with some additionl constraints or specific internal symmetries. This problem was tackled in the past e.g.~by specifying conditions for $L$ to generate a \emph{positive} or \emph{$k$-positive} semigroup \cite{Gorini1976,Chruscinski2011,Chruscinski2019}, or Kadison-Schwarz semigroup \cite{Lindblad1976}.


\bibliographystyle{unsrt}
\bibliography{DecomposableDynamicsGeneralBib}

\end{document}